\documentclass{article}
\usepackage{amsmath,amsthm,amsfonts,latexsym,amscd,amssymb,fancyhdr}
\newtheorem{theorem}{Theorem}[section]
\newtheorem{lemma}[theorem]{Lemma}
\newtheorem{corollary}[theorem]{Corollary}

\theoremstyle{definition}
\newtheorem{definition}[theorem]{Definition}

\theoremstyle{remark}
\newtheorem{remark}[theorem]{Remark}

\newcommand{\N}{\mathbb{N}}

\newcommand{\n}{\noindent}
\begin{document}

\title{A Note On Transversals}

\author{Vivek Kumar Jain
~\\ 
Department of Mathematics, Central University of Bihar \\ 
Patna (India) 800 014 \\
E-mail: jaijinenedra@gmail.com}

\date{}
\maketitle

\textbf{Abstract:} Let $G$ be a finite group and $H$ a core-free subgroup of $G$. We will show that if there exists a solvable, generating transversal of $H$ in $G$, then $G$ is a solvable group. Further, if $S$ is a generating transversal of $H$ in $G$ and $S$ has order $2$ invariant sub right loop $T$ such that the quotient $S/T$ is a group. Then $H$ is an elementary abelian 2-group.  
\vspace{1cm}

\smallskip


\vspace{1cm}

\section{Introduction}
Transversals play an important role in characterizing the group and the embedding of subgroup in the group. In \cite{tm}, Tarski monsters has been characterized with the help of transversals.  In \cite{p}, Lal shows that if all subgroups of a finitely generated solvable group are stable (any right transversals of a subgroup have isomorphic group torsion), then the group is nilpotent. In \cite{ict2}, it has been shown that if the isomorphism classes of transversals of a subgroup in a finite group is $3$, then the subgroup itself has index $3$. Converse of this fact is also true. In this paper we will characterize solvability of group in terms of a generating transversal of a core-free subgroup. Also we will propose some problems.

In \cite{rltr}, it has been shown that each right transversal is a right loop with respect to the operation induced by the operation of the group. Let us see how. 
 Let $G$ be a finite group and $H$ a proper core-free subgroup of $G$. A {\textit{right transversal}} is a subset of $G$ obtained by selecting one and only one element from each right coset of $H$ in $G$ and identity from the coset $H$. Now we will call it {\textit {transversal}} in place of right transversal. Suppose that $S$ is a transversal of $H$ in $G$. We define an operation $\circ $ on $S$ as follows: for $x,y \in S$, $\{x \circ y \}:=S \cap Hxy$.
It is easy to check that $(S, \circ )$ is a  right loop, that is the equation of the type $X \circ a=b$, $X$ is unknown and $a,b \in S$ has unique solution in $S$ and $(S, \circ)$ has both sided identity.
 In \cite{rltr}, it has been shown that for each right loop there exists a pair $(G,H)$ such that $H$ is a core-free subgroup of the group $G$ and the given right loop can be identified with a tansversal of $H$ in $G$. Not all transversals of a subgroup generate the group. But it is proved by Cameron in \cite{pjc}, that if a subgroup is core-free, then always there exists a transversal which generates the whole group. We call such a transversal as {\it generating transversal}. 

The difference between right loop and group is the associativity of operation. In \cite{rltr}, a notion of {\it group torsion} is introduced which measures the deviation of a right loop from being a group. Let us explain. 
Let $(S, \circ)$ be a right loop (identity denoted by $1$). Let $x,y,z \in S$. Define a map $f^S(y,z)$ from $S$ to $S$ as follows: $f^S(y,z)(x)$ is the unique solution of the equation $X \circ (y \circ z )=(x \circ y ) \circ z$, where $X$ is unknown. It is easy to verify that $f^S(y,z)$ is a bijective map. We denote by $G_S$ the subgroup of $Sym(S)$, the symmetric group on $S$ generated by the set $\{f^S(y,z) \mid y,z \in S \}$. This group is called {\textit{group torsion}} of $S$ \cite[Definition 3.1, p. 75]{rltr}. Further, it acts on $S$ through the action $\theta^S$ defined as: for $x \in S$ and $h \in G_S$; $x \theta^S h := h(x)$. Also note that right multiplication by an element of $S$ gives a bijective map from $S$ to $S$, that is an element of $Sym(S)$. 
The subgroup generated by this type of elements in $Sym(S)$ is denoted by $G_SS$ because $G_S$ is a subgroup of it and the right multiplication map associated with the elements of $S$ form a transversal of $G_S$ in $G_SS$. Note that if $H$ is core-free subgroup of a group $G$ and $S$ is a transversal of $H$ in $G$. Then $G\cong G_SS$ such that $H \cong G_S$ \cite[Proposition 3.10, p. 80]{rltr}.

In this paper we will prove following two results:

\begin{theorem}\label{1}
If a finite group has a solvable generating transversal with respect to a core-free subgroup, then  the group is solvable.
\end{theorem}
Note that by a solvable transversal we mean a transversal which is solvable with respect to induced right loop structure and solvability of right loop is introduced in  Definition \ref{s}.

\begin{theorem}\label{2}
Suppose that $S$ is a right loop and $T$ is its invariant sub right loop such that the quotient $S/T$ is a group. Then the group torsion of $S$ will be elementary abelian $2$-group. 
\end{theorem}

\section{Definitions and Lemmas}
Now for defining solvability of a right loop, we need following basic definitions and results.

\begin{definition}\cite{rps}
An equivalence relation $R$ on a right loop $S$ is called a congruence in $S$, if it is a sub right loop of $S \times S$. 
\end{definition}

\begin{definition}\cite{rps}
An equivalence class of identity of a congruence is called invariant sub right loop.
\end{definition}

\begin{remark}\label{r3}
Suppose that $I$ is an invariant sub right loop of a right loop $(S, \circ)$. Then $R=\{(x \circ y, y ) \mid x \in I, ~ \text ~ y \in S\}$ is a congruence in $S$ (for details see \cite[Theorem 2.7]{rps}) such that  the equivalence class of identity $R_1=I$. 
\end{remark}
 \begin{remark} \label{r4}
 If $R$ is a congruence on a right loop $(S,\circ)$ and $R_1$ denotes the equivalence class of identity, then the set $S/R_1:=\{ R_1 \circ x \mid x \in S \} $ together with operation $\circ_1$ defined as
 $(R_1\circ x) \circ_1 (R_1\circ y)=R_1 \circ (x\circ y)$, is a right loop. Now onwards we will use the operation of $S$ to denote the operation of $S/R_1$. Note that $S/R_1$ is also denoted as $S/R$.
 \end{remark}

\begin{lemma}\label{gl2}
Let $(S,\circ)$ be a right quasigroup with identity. Let $R$ be the congruence on $S$ generated by $\{(x,x\theta^Sf^S(y,z))~|x,y, z \in S\}$. Then $R$ is the smallest 
congruence on $S$ such that the quotient right loop $S/R$ is a group.
\end{lemma} 
 
\begin{proof} First, we will show that $S/R$ is a group. Let $1$ denote the identity of $S$. Since $R$ is a congruence on $S$, $R_1$ (equivalence class of $1$ under $R$) is an invariant sub right loop of $S$. So for showing that $S/R$ is a group, it is sufficient to show that the binary operation of $S/ R$ is associative. Let $x, y, z \in S$. Then
\n
\begin{eqnarray*}
 ((R_1\circ x)\circ (R_1\circ y))\circ (R_1\circ z) & = & (R_1\circ (x\circ y))\circ (R_1\circ z)\\
& = & R_1\circ ((x\circ y)\circ z)\\
& = & R_1\circ ( x \theta ^S f^S(y,z)\circ (y\circ z))\\
& = & R_1\circ (x \theta ^S f^S(y,z))\circ R_1\circ (y\circ z)\\
& = & R_1\circ (x \theta ^S f^S(y,z))\circ ((R_1\circ y)\circ(R_1\circ z))\\
& = & (R_1\circ x)\circ ((R_1\circ y) \circ (R_1\circ z))\\ 
& & (\text{since} ~(x,x \theta ^S f^S(y,z))\in  R).
 \end{eqnarray*}
Hence $S/R$ is a group.
Let $\phi:S \rightarrow S/{R}$ be the quotient homomorphism ($\phi(x)=R_1\circ x,~x\in S$).
 Let $H$ be a group with a homomorphism (of right loops) $\phi'~: ~ S\rightarrow H$. Since $\phi'$ is a homomorphism of right loops, $\phi'(S)$ is a sub right loop of $H$. 
 Further, since the binary operation on $\phi'(S)$ is associative, it is a subgroup of $H$.
It is easy to verify that Ker$\phi'$ is an invariant sub right loop, that is there exists a congruence $K$ on $S$ such that $K_1=Ker\phi'$ (by Remark \ref{r3}). By the Fundamental Theorem of homomorphisms for right loops there exists a unique one-one homomorphism $\bar{\phi}'~:S/K \rightarrow H$ such that $\bar{\phi}'\circ \nu={\phi}'$, where $\nu :S \rightarrow S/K$ is the quotient homomorphism ($\nu(x)=K_1\circ x,~x\in S$). Since $S/K$ is a group (being isomorphic to the subgroup ${\phi}'(S)$ of $H$), the associativity of its binary operation implies that $(x,x \theta ^S f^S(y,z))\in K$ for all $x,y,z \in S$. This implies $R \subseteq K$. This defines an onto homomorphism $\eta$ from $S/R$ to $S/K$ given by $\eta (R_1\circ x)=K_1\circ x$. Let ${\eta}'= \bar{\phi}'\circ \eta$. Then it follows easily that ${\eta}'$ is the unique homomorphism from $S/R$ to $H$ such that ${\eta}'\circ  \phi= {\phi}'$.
\end{proof}
\begin{remark}\label{gr3}
 Let $S$ and $R$ be as in the above Lemma. Let $T$ be a congruence
on $S$ containing $R$. Since $T_1/R_1$ is an invariant sub right loop of $S/R_1$, it is a normal subgroup of $S/R_1$. Thus $S/T$ is a group for it is
isomorphic to $S/R_1/(T_1/R_1)$.
\end{remark}
It is easy to prove following lemma.
\begin{lemma}\label{l3}
Let $(S, \circ)$ be a right loop. Let $L$ be the congruence on $S$ generated by $\{(x,x\theta^Sf^S(y,z))~|x,y\in S\} \cup \{(x \circ y, y\circ x) \mid x,y \in S\}$. Then $L$ is the smallest 
congruence on $S$ such that the quotient right loop $S/L$ is an abelian group.  
\end{lemma}

\begin{definition}
We define $S^{(1)}$ to be the smallest invariant sub right loop of $S$ such that $S/S^{(1)}$ is an abelian group that is, if there is another invariant sub right loop $N$ such that $S/N$ is an abelian group, then $S^{(1)} \subseteq N$.

We define $S^{(n)}$ by induction. Suppose $S^{(n-1)}$ is defined. Then $S^{(n)}$ is an invariant sub right loop of $S$ such that $S^{(n)}=(S^{(n-1)})^{(1)}$.
\end{definition}

Note that if $S$ is a group, then $S^{(1)}$ is the commutator subgroup and $S^{(n)}$ is the $n$th-commutator subgroup.

\begin{definition}\label{s}
We call $S$ solvable if  there exists an $n \in \N$ such that $S^{(n)}=\{1\}$.
\end{definition}

\begin{lemma}\label{a}
Let $G$ be a group, $H$ a subgroup of $G$ and $S$ a transversal of $H$ in $G$. Suppose that $N\trianglelefteq G $ containing $H$. Then 
$$ G/N=HS/N \cong S/{N\cap S}.$$
\end{lemma}
\begin{proof} Suppose that $\circ$ denotes the induced right loop operation on $S$. Consider the map $\psi :S \rightarrow HS/N$ defined as $\psi(x)=xN$. This is a homomorphism, for
\begin{eqnarray*}
\psi(x\circ y) & = & (x\circ y)N\\
~~~~~~~~~~& = & hxyN ~\text{for some $h \in H$}\\ 
~~~~~~~~~~& = & xyN~~(H \subseteq N)\\
~~~~~~~~~~& = & xNyN \\
~~~~~~~~~~& = & \psi(x) \psi(y).
\end{eqnarray*}     
Also, $Ker \psi =\{x \in S |~xN=N \}
=\{x \in S |~ x \in N \}=S \cap N.$
Since for $h\in H~$ and $~x\in S,$~we have $hxN=xN$ and $\psi (x)=xN$, $\psi$ is onto and so by the Fundamental Theorem of homomorphisms for right loops,
$S/{N\cap S} \cong HS/N $.
\end{proof}
 
 Let $G$ be a group, $H$ a subgroup and $S$ a transversal of $H$ in $G$. Suppose that $\circ$ is the induced right loop structure on $S$. We define a map $f: S \times S \rightarrow H$ as: for $x,y \in S$, $f(x,y):=xy(x \circ y)^{-1}$. We further define the action $\theta$ of $H$ on $S$ as $\{x \theta h\}:= S \cap Hxh$ where $h \in H$ and $x \in S$. With these notations it is easy to prove following lemma. 
 
 \begin{lemma}\label{l4}
 For $x,y,z \in S$, we have $x \theta^Sf^S(y,z)=x \theta f(y,z)$.
 
 \end{lemma} 
 
\begin{lemma}\label{b}
Let $H$ be a subgroup of a group $G$ and $S$ a transversal of $H$ in $G$. Let $T$ be the congruence on $S$ such that~$\{(x,x\theta h)~|~h \in H,x\in S \} \subseteq T $. Then $S/T$ is a group. Moreover, $N=HT_1 \trianglelefteq HS=G$ (and so $H \leqslant N$ and~ $N\cap S=T_1)$ and $G/N \cong S/T $, where $T_1$ denotes the equivalence class of 1 under $T$.
\end{lemma}

\begin{proof}~By Lemma \ref{l4}, $x\theta f(y,z)=x\theta^S f^S(y,z)$ for all $ x,y,z \in S$. Then by Lemma \ref{gl2}, $R \subseteq T$ and by Remark \ref{gr3}, $S/T$ is a group.
Let $\phi:G \rightarrow S/T$ be the map defined by $\phi(hx)=T_1\circ x,~h\in H,~x\in S$. This is a homomorphism, because for all $h_1,~h_2 \in H$ and $x_1,~x_2 \in S,$
\begin{eqnarray*} 
\phi(h_1x_1.h_2x_2) & = & \phi(h_1h(x_1{\theta} h_2\circ x_2))~\text{for some $h \in H$}\\
 & = & T_1\circ (x_1{\theta} h_2\circ x_2)\\
 & = & (T_1\circ (x_1{\theta} h_2))~\circ ~(T_1\circ x_2)\\
 & = & (T_1\circ x_1)~\circ~(T_1\circ x_2)~~~~~~~(\text{for}~(x_1{\theta} h_2,x_1) \in T)\\
 & = & \phi (h_1x_1)\phi (h_2x_2).
\end{eqnarray*}
Let $h\in H$ and $x\in S.$ Then $hx\in Ker \phi $ if and only if $x\in T_1.$ Hence $Ker\phi = HT_1=N(say)$. This proves the lemma.
\end{proof}

\section{Proof of Theorems}

\begin{proof}[Proof of Theorem \ref{1}]
Let $G$ be a finite group and $H$ a core-free subgroup of it. Suppose that $S$ is a generating transversal of $H$ in $G$. Then the group $G$ can be written as $HS$. By Lemma \ref{a}, $G/HG^{(1)} \cong S/S\cap HG^{(1)}$.
So, \begin{equation}\label{(i)} S^{(1)} \subseteq S\cap HG^{(1)}.
\end{equation}

By Lemma \ref{b}, $HS^{(1)}$ is normal subgroup of $G$.
Thus $G/HS^{(1)} =S/S^{(1)} $ (Lemma \ref{a}).
Since $S/S^{(1)} $ is abelian, $ G^{(1)} \subseteq HS^{(1)}.$ Thus  
\begin{equation} \label{(ii)} S\cap HG^{(1)}\subseteq S^{(1)}.
\end{equation} 

From \eqref{(ii)} and \eqref{(i)}, it is clear that \begin{equation} S\cap HG^{(1)} = S^{(1)} \end{equation}\label{(iii)(a)} and
\begin{equation}\label{(iii)(b)}HG^{(1)} =HS^{(1)}.\end{equation} 

We will use induction to prove that, $HS^{(n)}=  H(HS^{(n-1)})^{(1)}$ for $n \geq 1$ and by $S^{(0)}$ we mean $S$. For $n=1$, $HS^{(1)}= HG^{(1)}= H(HS^{(0)})^{(1)}$ (by \eqref{(iii)(b)}).
By induction, suppose that $HS^{(n-1)}= H(HS^{(n-2)})^{(1)}$.
Since  $S^{(n-1)}/S^{(n)}\cong HS^{(n-1)}/HS^{(n)}  $ is an abelian group, $(HS^{(n-1)})^{(1)} \subseteq HS^{(n)}$. 
Thus $ H(HS^{(n-1)})^{(1)}$ $ \subseteq HS^{(n)}$.

 Now, $HS^{(n-1)}/H(HS^{(n-1)})^{(1)} $ $\cong S^{(n-1)}/(S^{(n-1)} \cap H(HS^{(n-1)})^{(1)})$ (Lemma \ref{a}). 
So $S^{(n)} \subseteq S^{(n-1)} \cap H(HS^{(n-1)})^{(1)}=S \cap H(HS^{(n-1)})^{(1)}$. 
That is, \begin{equation}\label{(*)} HS^{(n)}= H(HS^{(n-1)})^{(1)}~\text{for all}~ n \geq 1.\end{equation}  
Now  \eqref{(*)} implies that \begin{equation}\label{(**)} HS^{(n)}=H(HS^{(n-1)})^{(1)}  \supseteq H(H(HS^{(n-2)})^{(1)})^{(1)} \supseteq H(HS^{(n-2)})^{(2)} . \end{equation}  
Proceeding inductively, we have $ HS^{(n)}\supseteq H(HS)^{(n)}=HG^{(n)}$.
Suppose that $S$ is solvable, that is there exists $n \in \mathbb{N}$ such that $S^{(n)}=\{1\}$. Then $(G)^{(n)} \subseteq H$. Since $(G)^{(n)}$ is normal subgroup of $G$ contained in $H$, so $(G)^{(n)}=\{1\}$. This proves the theorem.
\end{proof}
 
 Converse of the above theorem is not true. For example take $G$ to be  the Symmetric group on three symbols and $H$ to be any two order subgroup of it. Then $H$ has no solvable generating transversal but we know that $G$ is solvable.  Following is the easy consequence of the above theorem.
 
 \begin{corollary}
 If $S$ is a finite solvable right loop, then $G_SS$ is solvable.
 \end{corollary}

\begin{proof}[Proof of Theorem \ref{2}]
Suppose that $T=\{1, t\}$. Since $T$ is an invariant sub right loop, so $t$ is fixed by all the elements of $G_S$. 
Consider the equation $(x \circ y ) \circ z = (x \theta^S f^S(y,z)) \circ (y \circ z)$. 
Since $S/T$ has associativity, so  $T \circ ( (x \circ y ) \circ z) = (T \circ x \theta^S f^S(y,z)) \circ ((T \circ y) \circ (T \circ z))$. Since $S/T$ is a group, so by cancellation law we have $T\circ x=T \circ (x \theta^S f^S(y,z))$. This implies $\{x, t \circ x \}=\{x \theta^S f^S(y,z), t \circ x \theta^S f(y,z) \}$. Note that for each $y,z \in S$, an element $x \in S$ is either fixed by $f^S(y,z)$ or $x \theta^S f^S(y,z)=f^S(y,z)(x)=t \circ x$. Suppose that $f^S(y,z)(x)=t \circ x$. Since $f^S(y,z)$ is bijective, so $f^S(y,z)(t \circ x) \neq t \circ x$. Then $f^S(y,z)(t \circ x)= t\circ (t \circ x)=(t \theta^S f^S(y,z)^{-1} \circ t) \circ x=(t \circ t )\circ x =x$. Thus either $x$ is fixed by $f^S(y,z)$ or there is a transposition $(x, t \circ x)$ in the cycle decomposition of $f^S(y,z)$. That is $f^S(y,z)$ can be written as a product of disjoint transpositions. This clearly implies that $G_S= \langle \{f^S(y,z) \mid y,z \in S \} \rangle $ is an elementary abelian 2-group. This proves the theorem.
\end{proof}

Following is a known result (can be proved independently) follows from Theorem \ref{2}. 

\begin{corollary}
Let $G$ be a finite group and $H$ a core-free subgroup contained in a normal subgroup $N$ such that the index of $H$ in $N$ is $2$. Then $N$ is an elementary abelian $2$-group.
\end{corollary}

\section{Some Problems}
It is an interesting problem to find condition on  a right loop $S$ such that $G_SS$ holds some group theoretical property. For example, what is the minimum condition on right loop $S$ for which $G_SS$ is a   nilpotent group?


\begin{thebibliography}{99}



\bibitem[1]{pjc} Cameron, P. J. {{ Generating a group by a transversal}}. preprint available at http://www.maths.qmul.ac.uk/~pjc/preprints/transgenic.pdf 

\bibitem[2]{ict2} Jain, V.K., Shukla, R.P. (2011). { On the isomorphism classes of transversals II}, {\it Comm. Algebra} 39:2024-2036.

\bibitem[3]{rltr} Lal, R. (1996).
{ Transversals in Groups}.
{\it J. Algebra} 181:70-81.

\bibitem[4]{p} Lal, R.  {\it Some Problems on Dedekind-Type Groups}, J. Algebra 181 (1996), 223-234.


\bibitem[5]{tm} Shukla, R. P. 
{\it A characterization of Tarski monsters}, Indian J. Pure Appl. Math 36 (12) (2005), 673-678.






\bibitem[3]{rps} Shukla, R. P. (1995) {{ Congruences in right quasigroups and general extensions}}. {\it Comm. Algebra} 23:2679-2695.





\end{thebibliography}
\end{document}